\documentclass[]{interact}

\usepackage{epstopdf}
\usepackage[caption=false]{subfig}
\usepackage{xcolor}
\usepackage{amsmath}
\usepackage{amsmath,cases}
\usepackage{mathtools}
\usepackage{xcolor}
\usepackage[numbers,sort&compress]{natbib}
\usepackage{mathtools}
\bibpunct[, ]{[}{]}{,}{n}{,}{,}
\renewcommand\bibfont{\fontsize{10}{12}\selectfont}
\makeatletter
\def\NAT@def@citea{\def\@citea{\NAT@separator}}
\makeatother

\theoremstyle{plain}
\newtheorem{theorem}{Theorem}[section]
\newtheorem{lemma}[theorem]{Lemma}

\newtheorem{proposition}[theorem]{Proposition}

\theoremstyle{definition}
\newtheorem{definition}[theorem]{Definition}

\theoremstyle{remark}
\newtheorem{remark}{Remark}

\DeclareMathOperator*{\esup}{ess\ sup}

\begin{document}

\articletype{Research Paper}

\title{Variational and nonvariational solutions for double phase variable exponent problems}

\author{\name{Mustafa Avci\thanks{CONTACT M.~Avci. Email:  mavci@athabascau.ca (primary) \& avcixmustafa@gmail.com}}
\affil{Faculty of Science and Technology, Applied Mathematics, Athabasca University, AB, Canada}}

\maketitle

\begin{abstract}
In this article, we examine two double-phase variable exponent problems, each formulated within a distinct framework. The first problem is non-variational, as the nonlinear term may depend on the gradient of the solution. The first main result establishes an existence property from the nonlinear monotone operator theory given by Browder and Minty. The second problem is set up within a variational framework, where we employ a well-known critical point result by Bonanno and Chinn\`{\i}. In both cases, we demonstrate the existence of at least one nontrivial solution. To illustrate the practical application of the main results, we provide examples for each problem.
\end{abstract}


\begin{amscode}
35A01; 35A16; 35D30; 35Q56
\end{amscode}

\section{Introduction}
In this article, we study the following two double phase variable exponent problems each is set up in a different framework to correspond to different physical phenomena.\\
The first problem considered is in the nonvariational structure since we let the nonlinearity $f$ have gradient dependence (so-called convection terms)
\begin{equation}\label{e1.1}
\begin{cases}
\begin{array}{rlll}
&-\mathrm{div}(|\nabla u|^{p(x)-2}\nabla u+\mu(x)|\nabla u|^{q(x)-2}\nabla u)\\
&+\alpha(x)\left(\frac{|u|^{ r(x)}}{r(x)}-\gamma(x)\right)|u|^{r(x)-2}u=f(x,u,\nabla u) \quad \text{ in }\Omega, \\
& \qquad \qquad \qquad \qquad \quad \quad \quad \quad \quad u=0  \quad \text{ on }\partial \Omega,\tag{$\mathcal{P}$}
\end{array}
\end{cases}
\end{equation}
where $\Omega$ is a bounded domain in $\mathbb{R}^N$ $(N\geq2)$ with Lipschitz boundary; $p,q,r \in C_+(\overline{\Omega })$; $f \in W_0^{1,\mathcal{H}}(\Omega)^*$; $0\leq \mu(\cdot)\in L^\infty(\Omega)$; $\alpha,\gamma\in L^{\infty}(\Omega)$ with $\inf_{x\in \Omega}\alpha(x),\gamma(x)>0$. By applying the existence results (Lemma \ref{Lem:3.1}) of Browder \cite{browder1963nonlinear} and Minty \cite{minty1963monotonicity}, we obtain the existence of at least one nontrivial non-variational solution.\\

The second problem is given in the variational structure
\begin{equation}\label{e1.1a}
\begin{cases}
\begin{array}{rlll}
&-\mathrm{div}(|\nabla u|^{p(x)-2}\nabla u+\mu(x)|\nabla u|^{q(x)-2}\nabla u)=\lambda f(x,u) \quad \text{ in }\Omega, \\
& \qquad \qquad \qquad \qquad \qquad \qquad \quad \qquad \qquad u=0  \quad \text{ on }\partial \Omega,\tag{$\mathcal{P_\lambda}$}
\end{array}
\end{cases}
\end{equation}
where $\Omega$ is a bounded domain in $\mathbb{R}^N$ $(N\geq2)$ with Lipschitz boundary; $p,q\in C_+(\overline{\Omega })$; $f$ is a Carathéodory function; and $\lambda>0$ is a real paramater. We employ a well-known critical point result by Bonanno and Chinn\`{\i} \cite{bonanno2014existence} to obtain the existence of at least one nontrivial variational solution.\\

Equations of the form (\ref{e1.1a}) appear in models involving materials with non-uniform (anisotropic) properties, fluid mechanics, image processing, and elasticity in non-homogeneous materials. Consequently, problem (\ref{e1.1a}) can be used to model various real-world phenomena, primarily due to the presence of the operator
\begin{equation}\label{e1.2a}
\mathrm{div}\left(|\nabla u|^{p(x)-2}\nabla u+\mu(x)|\nabla u|^{q(x)-2}\nabla u\right)
\end{equation}
 which governs anisotropic and heterogeneous diffusion associated with the energy functional
\begin{equation}\label{e1.2b}
u\to \int_\Omega\left(\frac{|\nabla u|^{p(x)}}{p(x)}+\mu(x)\frac{|\nabla u|^{q(x)}}{q(x)}\right)dx
\end{equation}
is called a "double phase" operator because it encapsulates two different types of elliptic behavior within the same framework. The functional of the form (\ref{e1.2b}) was first introduced in \cite{zhikov1987averaging} for constant exponents. Since then, numerous studies have explored this topic due to its wide applicability across various disciplines (e.g., \cite{baroni2015harnack,baroni2018regularity,colombo2015bounded,colombo2015regularity,marcellini1991regularity,marcellini1989regularity,galewski2024variational,bu2022p,lapa2015no,el2023existence,albalawi2022gradient}).\\

The paper is organised as follows. In Section 2, we first provide some background for the theory of variable Sobolev spaces $W_{0}^{1,p(x)}(\Omega)$ and the Musielak-Orlicz Sobolev space $W_0^{1,\mathcal{H}}(\Omega)$, and then obtain an auxiliary result: Lemma \ref{Lem:3.3ba}. In Section 3, we set up the nonlinear operator equation corresponding to the problem (\ref{e1.1}), and obtain a nontrivial non-variational solution. In Section 4, obtain a nontrivial variational solution for the problem (\ref{e1.1a}). At the end of Sections 3 and 4, we present examples for each problem to demonstrate the practical application of the main results. As far as we're concerned, this paper is the first to address both variational and non-variational double phase variable exponent problems  within a single article.

\section{Mathematical Background and Auxiliary Results}

We start with some basic concepts of variable Lebesgue-Sobolev spaces. For more details, and the proof of the following propositions, we refer the reader to \cite{cruz2013variable,diening2011lebesgue,edmunds2000sobolev,fan2001spaces,radulescu2015partial}.
\begin{equation*}
C_{+}\left( \overline{\Omega }\right) =\left\{p\in C\left( \overline{\Omega }\right): \text{ }p\left( x\right) >1\text{ for all\ }x\in
\overline{\Omega }\right\} .
\end{equation*}
For $p\in C_{+}( \overline{\Omega }) $ denote
\begin{equation*}
p^{-}:=\underset{x\in \overline{\Omega }}{\min }p( x) \leq p( x) \leq p^{+}:=\underset{x\in \overline{\Omega }}{\max}p(x) <\infty .
\end{equation*}
For any $p\in C_{+}\left( \overline{\Omega }\right) $, we define \textit{the
variable exponent Lebesgue space} by
\begin{equation*}
L^{p(x)}(\Omega) =\left\{ u\mid u:\Omega\rightarrow\mathbb{R}\text{ is measurable},\int_{\Omega }|u(x)|^{p(x) }dx<\infty \right\}.
\end{equation*}
Then, $L^{p(x)}(\Omega)$ endowed with the norm
\begin{equation*}
|u|_{p(x)}=\inf \left\{ \lambda>0:\int_{\Omega }\left\vert \frac{u(x)}{\lambda }
\right\vert ^{p(x)}dx\leq 1\right\} ,
\end{equation*}
becomes a Banach space.\\
The convex functional $\rho :L^{p(x) }(\Omega) \rightarrow\mathbb{R}$ defined by
\begin{equation*}
\rho(u) =\int_{\Omega }|u(x)|^{p(x)}dx,
\end{equation*}
is called modular on $L^{p(x) }(\Omega)$.

\begin{proposition}\label{Prop:2.2} If $u,u_{n}\in L^{p(x) }(\Omega)$, we have
\begin{itemize}
\item[$(i)$] $|u|_{p(x) }<1 ( =1;>1) \Leftrightarrow \rho(u) <1 (=1;>1);$
\item[$(ii)$] $|u|_{p( x) }>1 \implies |u|_{p(x)}^{p^{-}}\leq \rho(u) \leq |u|_{p( x) }^{p^{+}}$;\newline
$|u|_{p(x) }\leq1 \implies |u|_{p(x) }^{p^{+}}\leq \rho(u) \leq |u|_{p(x) }^{p^{-}};$
\item[$(iii)$] $\lim\limits_{n\rightarrow \infty }|u_{n}-u|_{p(x)}=0\Leftrightarrow \lim\limits_{n\rightarrow \infty }\rho (u_{n}-u)=0$.
\end{itemize}
\end{proposition}

\begin{proposition}\label{Prop:2.2bb}
Let $p_1(x)$ and $p_2(x)$ be measurable functions such that $p_1\in L^{\infty}(\Omega )$ and $1\leq p_1(x)p_2(x)\leq \infty$ for a.e. $x\in \Omega$. Let $u\in L^{p_2(x)}(\Omega ),~u\neq 0$. Then
\begin{itemize}
\item[$(i)$] $\left\vert u\right\vert _{p_1(x)p_2(x)}\leq 1\text{\ }\Longrightarrow
\left\vert u\right\vert _{p_1(x)p_2(x))}^{p_1^{+}}\leq \left\vert \left\vert
u\right\vert ^{p_1(x)}\right\vert _{p_2(x)}\leq \left\vert
u\right\vert _{p_1(x)p_2(x)}^{p^{-}}$
\item[$(ii)$] $\left\vert u\right\vert _{p_1(x)p_2(x)}> 1\ \Longrightarrow \left\vert
u\right\vert _{p_1(x)p_2(x)}^{p_1^{-}}\leq \left\vert \left\vert u\right\vert
^{p_1(x)}\right\vert _{p_2(x) }\leq \left\vert u\right\vert
_{p_1(x)p_2(x)}^{p_1^{+}}$
\item[$(iii)$] In particular, if $p_1(x)=p$ is constant then
\begin{equation*}
\left\vert \left\vert u\right\vert ^{p}\right\vert _{p_2(x)}=\left\vert
u\right\vert _{pp_2(x)}^{p}.
\end{equation*}
\end{itemize}
\end{proposition}
The variable exponent Sobolev space $W^{1,p(x)}( \Omega)$ is defined by
\begin{equation*}
W^{1,p(x)}( \Omega) =\{u\in L^{p(x) }(\Omega) : |\nabla u| \in L^{p(x)}(\Omega)\},
\end{equation*}
with the norm
\begin{equation*}
\|u\|_{1,p(x)}=|u|_{p(x)}+|\nabla u|_{p(x)},\\\ \forall u\in W^{1,p(x)}(\Omega).
\end{equation*}
The space $W_{0}^{1,h(x)}(\Omega)$ is defined as the closure of $C_{0}^{\infty }(\Omega )$ in $W^{1,p(x)}(\Omega)$.\\
\begin{proposition}\label{Prop:2.4} If $p\in C_+(\overline{\Omega })$ and $ p^{+}<\infty$, then the spaces $L^{p(x) }( \Omega)$, $W^{1,p(x)}(\Omega)$, and $W_{0}^{1,p(x)}(\Omega)$ are separable and reflexive Banach spaces.
\end{proposition}
Furthermore,  Poincar\'{e} inequality holds in $W_{0}^{1,p(x)}(\Omega)$ \cite{fan2001spaces}; that is, there exists a
positive constant $c$ independent of $u$ such that
\begin{equation*}
|u|_{p(x)}\leq c|\nabla u|_{p(x)},\quad \forall u\in W_{0}^{1,p(x)}(\Omega),
\end{equation*}
which implies that $|\nabla u|_{p(x)}$ is an equivalent norm in $W_{0}^{1,p(x)}(\Omega)$. Therefore, on $W_{0}^{1,p(x)}(\Omega)$ we can define an equivalent norm $\|\cdot\|$ such that
\begin{equation*}
\|u\| =|\nabla u|_{p(x)}.
\end{equation*}

\begin{proposition}\label{Prop:2.5} Let $q\in C(\overline{\Omega })$. If $1\leq q(x) <p^{\ast }(x)$ for all $x\in
\overline{\Omega }$, then the embeddings $W^{1,p(x)}(\Omega) \hookrightarrow L^{q(x) }(\Omega)$ and $W_0^{1,p(x)}(\Omega) \hookrightarrow L^{q(x) }(\Omega)$  are compact and continuous, where
$p^{\ast }( x) =\left\{\begin{array}{cc}
\frac{Np(x) }{N-p( x) } & \text{if }p(x)<N, \\
+\infty & \text{if }p( x) \geq N.
\end{array}
\right. $
\end{proposition}

Throughout the paper, we assume the following.
\begin{itemize}
\item[$(H_1)$] $p,q \in C_+(\overline{\Omega})$, $p(x)<N$ and $p(x)<q(x)$ for all $x\in \overline{\Omega }$ with $q^+<\min\{p^*(x),N\}$ and $r^+<\frac{p^*(x)}{2}$;
\item[$(H_2)$] $\mu\in L^\infty(\Omega)$ such that $\mu(\cdot)\geq 0$.
\end{itemize}
To address problems (\ref{e1.1}) and (\ref{e1.1a}), it is necessary to utilize the theory of the Musielak-Orlicz Sobolev space $W_0^{1,\mathcal{H}}(\Omega)$. Therefore, we subsequently introduce the double-phase operator, the Musielak-Orlicz space, and the Musielak-Orlicz Sobolev space in turn.\\

Let $\mathcal{H}:\Omega\times [0,\infty]\to [0,\infty]$ be the nonlinear function, i.e. the \textit{double phase operator}, defined by
\[
\mathcal{H}(x,t)=t^{p(x)}+\mu(x)t^{q(x)}\ \text{for all}\ (x,t)\in \Omega\times [0,\infty].
\]
Then the corresponding modular $\rho_\mathcal{H}(\cdot)$ is given by
\[
\displaystyle\rho_\mathcal{H}(u)=\int_\Omega\mathcal{H}(x,|u|)dx=
\int_\Omega\left(|u|^{p(x)}+\mu(x)|u|^{q(x)}\right)dx.
\]
The \textit{Musielak-Orlicz space} $L^{\mathcal{H}}(\Omega)$, is defined by
\[
L^{\mathcal{H}}(\Omega)=\left\{u:\Omega\to \mathbb{R}\,\, \text{measurable}:\,\, \rho_{\mathcal{H}}(u)<+\infty\right\},
\]
endowed with the Luxemburg norm
\[
\|u\|_{\mathcal{H}}=\inf\left\{\zeta>0: \rho_{\mathcal{H}}\left(\frac{u}{\zeta}\right)\leq 1\right\}.
\]
Analogous to Proposition \ref{Prop:2.2}, there are similar relationship between the modular $\rho_{\mathcal{H}}(\cdot)$ and the norm $\|\cdot\|_{\mathcal{H}}$, see \cite[Proposition 2.13]{crespo2022new} for a detailed proof.

\begin{proposition}\label{Prop:2.2a}
Assume $(H_1)$ hold, and $u\in L^{\mathcal{H}}(\Omega)$. Then
\begin{itemize}
\item[$(i)$] If $u\neq 0$, then $\|u\|_{\mathcal{H}}=\lambda\Leftrightarrow \rho_{\mathcal{H}}(\frac{u}{\zeta})=1$,
\item[$(ii)$] $\|u\|_{\mathcal{H}}<1\ (\text{resp.}\ >1, =1)\Leftrightarrow \rho_{\mathcal{H}}(\frac{u}{\zeta})<1\ (\text{resp.}\ >1, =1)$,
\item[$(iii)$] If $\|u\|_{\mathcal{H}}<1\Rightarrow \|u\|_{\mathcal{H}}^{q^+}\leq \rho_{\mathcal{H}}(u)\leq \|u\|_{\mathcal{H}}^{p^-}$,
\item[$(iv)$]If $\|u\|_{\mathcal{H}}>1\Rightarrow \|u\|_{\mathcal{H}}^{p^-}\leq \rho_{\mathcal{H}}(u)\leq \|u\|_{\mathcal{H}}^{q^+}$,
\item[$(v)$] $\|u\|_{\mathcal{H}}\to 0\Leftrightarrow \rho_{\mathcal{H}}(u)\to 0$,
\item[$(vi)$]$\|u\|_{\mathcal{H}}\to +\infty\Leftrightarrow \rho_{\mathcal{H}}(u)\to +\infty$,
\item[$(vii)$] $\|u\|_{\mathcal{H}}\to 1\Leftrightarrow \rho_{\mathcal{H}}(u)\to 1$,
\item[$(viii)$] If $u_n\to u$ in $L^{\mathcal{H}}(\Omega)$, then $\rho_{\mathcal{H}}(u_n)\to\rho_{\mathcal{H}}(u)$.
\end{itemize}
\end{proposition}

\medskip

The \textit{Musielak-Orlicz Sobolev space} $W^{1,\mathcal{H}}(\Omega)$ is defined by
\[
W^{1,\mathcal{H}}(\Omega)=\left\{u\in L^{\mathcal{H}}(\Omega):
|\nabla u|\in L^{\mathcal{H}}(\Omega)\right\},
\]
and equipped with the norm
\[
\|u\|_{1,\mathcal{H}}=\|\nabla u\|_{\mathcal{H}}+\|u\|_{\mathcal{H}},
\]
where $\|\nabla u\|_{\mathcal{H}}=\|\,|\nabla u|\,\|_{\mathcal{H}}$.\\
\medskip
\noindent
The space $W_0^{1,\mathcal{H}}(\Omega)$ is defined as the closure of $C_{0}^{\infty }(\Omega )$ in $W^{1,\mathcal{H})}(\Omega)$.
Note also that $L^{\mathcal{H}}(\Omega), W^{1,\mathcal{H}}(\Omega)$ and $W_0^{1,\mathcal{H}}(\Omega)$ are reflexive Banach spaces \cite[Proposition 2.12]{crespo2022new}.\\

\medskip
\noindent
We now present the following embedding relations given in \cite[Proposition 2.16]{crespo2022new}.
\begin{proposition}\label{Prop:2.7a}
Assume that $(H_1)$ and $(H_2)$ hold. Then the following embeddings hold:
\begin{itemize}
\item[$(i)$] $L^{\mathcal{H}}(\Omega)\hookrightarrow L^{h(\cdot)}(\Omega), W^{1,\mathcal{H}}(\Omega)\hookrightarrow W^{1,h(\cdot)}(\Omega)$, $W_0^{1,\mathcal{H}}(\Omega)\hookrightarrow W_0^{1,h(\cdot)}(\Omega)$ are continuous for all $h\in C(\overline{\Omega})$ with $1\leq h(x)\leq p(x)$ for all $x\in \overline{\Omega}$.
\item[$(ii)$] $W^{1,\mathcal{H}}(\Omega)\hookrightarrow L^{h(\cdot)}(\Omega)$ and $W_0^{1,\mathcal{H}}(\Omega)\hookrightarrow L^{h(\cdot)}(\Omega)$ are compact for all $h\in C(\overline{\Omega})$ with $1\leq h(x)< p^*(x)$ for all $x\in \overline{\Omega}$.
\end{itemize}
\end{proposition}

\begin{proposition}\label{Prop:2.7cd}\cite{crespo2022new}
Assume that $(H_1)$ and $(H_2)$ hold. Then the following hold:
\begin{itemize}
\item[$(i)$] The embedding $W^{1,\mathcal{H}}(\Omega)\hookrightarrow L^{\mathcal{H}}(\Omega)$ is compact;
\item[$(ii)$] There exists a constant $c>0$ independent of $u$ such that
\[
\|u\|_{\mathcal{H}}\leq c\|\nabla u\|_{\mathcal{H}}, \quad \forall u \in W_0^{1,\mathcal{H}}(\Omega).
\]
\end{itemize}
\end{proposition}
\medskip
\noindent
As a conclusion of Proposition \ref{Prop:2.7cd}, the space $W_0^{1,\mathcal{H}}(\Omega)$ can be equipped with an equivalent norm $\|\cdot\|_{1,\mathcal{H},0}$ given by
\[
\|u\|_{1,\mathcal{H},0}=\|\nabla u\|_{\mathcal{H}}, \quad \forall u \in W_0^{1,\mathcal{H}}(\Omega).
\]
\begin{proposition}\label{Prop:2.7}
For the convex functional $\Phi(u):=\int_\Omega\left(\frac{|\nabla u|^{p(x)}}{p(x)}+\mu(x)\frac{|\nabla u|^{q(x)}}{q(x)}\right)dx$, we have $\Phi \in C^{1}(W_0^{1,\mathcal{H}}(\Omega),\mathbb{R})$ with the derivative
$$
\langle\Phi^{\prime}(u),\varphi\rangle=\int_{\Omega}(|\nabla u|^{p(x)-2}\nabla u+\mu(x)|\nabla u|^{q(x)-2}\nabla u)\cdot\nabla \varphi dx,
$$
for all  $u, \varphi \in W_0^{1,\mathcal{H}}(\Omega)$, where $\langle \cdot, \cdot\rangle$ is the dual pairing between $W_0^{1,\mathcal{H}}(\Omega)$ and its dual $W_0^{1,\mathcal{H}}(\Omega)^{*}$ \cite{crespo2022new}.
\end{proposition}
\begin{remark}\label{Rem:2.1a}
Notice that
\begin{itemize}
\item[$(i)$] If $\|\nabla u\|_{\mathcal{H}}<1\Rightarrow \frac{1}{q^+}\|\nabla u\|_{\mathcal{H}}^{q^+}\leq \Phi(u)\leq \frac{1}{p^-}\|\nabla u\|_{\mathcal{H}}^{p^-}$,
\item[$(ii)$]If $\|\nabla u\|_{\mathcal{H}}>1\Rightarrow \frac{1}{q^+}\|\nabla u\|_{\mathcal{H}}^{p^-}\leq \Phi(u)\leq \frac{1}{p^-}\|\nabla u\|_{\mathcal{H}}^{q^+}$.
\end{itemize}
\end{remark}
\begin{lemma}\label{Lem:3.3ba} For the functional $\mathcal{L}:W_0^{1,\mathcal{H}}(\Omega)\rightarrow \mathbb{R}$ defined by
\newline $\mathcal{L}(u):=\int_{\Omega}\frac{\alpha(x)}{2}\left(\frac{|u|^{r(x)}}{r(x)}-\gamma(x)\right)^{2}dx$ it holds:
\begin{itemize}
\item[$(i)$] $\mathcal{L} \in L^{1}(\Omega)$ such that
\begin{align}\label{e3.12fab}
\mathcal{L}(u) &\leq \frac{C|\alpha|_{\infty}}{(r^-)^2}\|u\|_{1,\mathcal{H},0}^{2r_M},\,\ \forall u \in W_0^{1,\mathcal{H}}(\Omega),
\end{align}
where $r_M:=\max\{r^-,r^+\}$, and $C>0$ is the embedding constant.
\item[$(ii)$] $\mathcal{L}$ is convex.
\item[$(iii)$] $\mathcal{L}$  is of class $C^{1}(W_0^{1,\mathcal{H}}(\Omega),\mathbb{R})$ and its derivative $\mathcal{L}^{\prime}:W_0^{1,\mathcal{H}}(\Omega)\rightarrow W_0^{1,\mathcal{H}}(\Omega)^*$ given by
\begin{equation}\label{e3.12ab}
  \langle \mathcal{L}^{\prime}(u),\varphi \rangle=\int_{\Omega}\alpha(x)\left(\frac{|u|^{r(x)}}{r(x)}-\gamma(x)\right)|u|^{r(x)-2}u \varphi dx,
\end{equation}
is monotonically increasing for $u, \varphi \in W_0^{1,\mathcal{H}}(\Omega)$.
\end{itemize}
\end{lemma}
\begin{proof}
$(i)$ Since $\gamma\in L^{\infty}(\Omega)$, there exists a real number $M_0>0$ such that
\[
 |\gamma|_{\infty}=\esup_{x\in \Omega}|\gamma(x)|=\{M\geq 0: |\gamma(x)|\leq M\,\, \text{for a.e.} x \in \Omega\}=M_0.
\]
If we partition $\Omega$ as
\[
\Omega_0:=\{x \in \Omega: |\gamma(x)|> M_0\}\quad \text{and} \quad \Omega\setminus\Omega_0:=\{x \in \Omega: |\gamma(x)|\leq M_0\},
\]
then the set $\Omega_0$  is of measure zero. Therefore
\begin{align*}
\int_{\Omega}\alpha(x)\left(\frac{|u|^{r(x)}}{r(x)}-\gamma(x)\right)^{2}dx&=\int_{\Omega\setminus\Omega_0}\alpha(x)\left(\frac{|u|^{r(x)}}{r(x)}-|\gamma(x)|\right)^{2}dx \nonumber \\
&\leq \int_{\Omega\setminus\Omega_0}\alpha(x)\left(\frac{|u|^{r(x)}}{r(x)}\right)^{2}dx \leq \frac{|\alpha|_{\infty}}{(r^-)^2}\int_{\Omega}|u|^{2r(x)}dx.
\end{align*}
Using the embedding $W_0^{1,\mathcal{H}}(\Omega)\hookrightarrow L^{2r(x)}(\Omega)$ gives the desired result \eqref{e3.12fab}.

\medskip
\noindent
$(ii)$ Using the same reasoning as with \cite[Proposition 2.7]{avci2019positive}), we shall show that $\mathcal{L}$ is convex if
\begin{equation*}
\mathcal{L}((1-\varepsilon)u+\lambda v)< (1-\varepsilon)\tau_1+\varepsilon \tau_2, \,\,\,\, 0<\varepsilon<1
\end{equation*}
whenever $\mathcal{L}(u)<\tau_1$ and $\mathcal{L}(v)<\tau_2$, for $u, v \in W_0^{1,\mathcal{H}}(\Omega)$ and $\tau_1,\tau_2 \in (0,\infty)$.\\
Since $(1-\varepsilon)u+\varepsilon v \in W_0^{1,\mathcal{H}}(\Omega)$, from part $(i)$, we have
\begin{align}\label{e3.12abd}
\mathcal{L}((1-\varepsilon)u+\varepsilon v)&<\frac{|\alpha|_{\infty}}{(r^-)^2}\int_{\Omega}|(1-\varepsilon)u+\varepsilon v|^{2r^+}dx \nonumber\\
&<\frac{2^{2r^+-1}|\alpha|_{\infty}(1-\varepsilon)}{(r^-)^2}\int_{\Omega}|u|^{2r^+}dx +\frac{2^{2r^+-1}|\alpha|_{\infty}\varepsilon }{(r^-)^2}\int_{\Omega}|v|^{2r^+}dx \nonumber\\
&< (1-\varepsilon)\tau_1+\varepsilon \tau_2,
\end{align}
where
$
\mathcal{L}(u)<\frac{2^{2r^+-1}|\alpha|_{\infty}}{(r^-)^2}\int_{\Omega}|u|^{2r^+}dx:=\tau_1,
$
and
$
\mathcal{L}(v)<\frac{2^{2r^+-1}|\alpha|_{\infty}}{(r^-)^2}\int_{\Omega}|v|^{2r^+}dx:=\tau_2.
$
Therefore $\mathcal{L}$ is convex.\\

$(iii)$  As for $\mathcal{L}$ being of class $C^{1}(W_0^{1,\mathcal{H}}(\Omega),\mathbb{R})$ with the derivative $\mathcal{L}^{\prime}:W_0^{1,\mathcal{H}}(\Omega)\rightarrow W_0^{1,\mathcal{H}}(\Omega)^*$
\begin{equation}\label{e3.12ab}
  \langle \mathcal{L}^{\prime}(u),\varphi \rangle=\int_{\Omega}\alpha(x)\left(\frac{|u|^{r(x)}}{r(x)}-\gamma(x)\right)|u|^{r(x)-2}u \varphi dx,
\end{equation}
was proved by Avci et al. in \cite[Lemma 3.5]{avci2019nonlocal}.\\
Now, we show that $\mathcal{L}^{\prime}$ is a monotonically increasing operator. To do so, let $u,v \in W_0^{1,\mathcal{H}}(\Omega)$ and $0<\sigma<1$. Since $\mathcal{L}$ is convex on $W_0^{1,\mathcal{H}}(\Omega)$, we have
\begin{align}\label{e3.11kk}
\frac{\mathcal{L}(u+\sigma (v-u))-\mathcal{L}(u)}{\sigma}& \leq \mathcal{L}(v)-\mathcal{L}(u).
\end{align}
Taking into account that $\mathcal{L}$ is of class $C^{1}(W_0^{1,\mathcal{H}}(\Omega),\mathbb{R})$, it reads
\begin{align}\label{e3.11kln}
&\langle \mathcal{L}^{\prime}(u), v-u \rangle \leq \mathcal{L}(v)-\mathcal{L}(u).
\end{align}
Applying the same argument gives
\begin{align}\label{e3.11kkm}
\langle \mathcal{L}^{\prime}(v), u-v \rangle &\leq  \mathcal{L}(u)- \mathcal{L}(v).
\end{align}
Using the inequalities (\ref{e3.11kln})and (\ref{e3.11kkm}) together gives
\begin{equation}\label{e3.11m}
\langle \mathcal{L}^{\prime}(u)-\mathcal{L}^{\prime}(v), u-v \rangle \geq 0,\,\,\ \text{for all}\,\, u, v \in W_0^{1,\mathcal{H}}(\Omega).
\end{equation}
\end{proof}

\section{Nonvariational problem}

We consider problem (\ref{e1.1}) being nonvariational since the nonlinear term $f$ may depend on the gradient of the solution. Our first main result demonstrates an existence property for problem (\ref{e1.1}) using the nonlinear monotone operator theory.
\begin{definition}\label{Def:3.1} A function $u\in W_0^{1,\mathcal{H}}(\Omega)$ is called a weak solution to problem (\ref{e1.1}) if for given any $f\in W_0^{1,\mathcal{H}}(\Omega)^*$, it holds
\begin{align}\label{e3.1}
&\int_{\Omega}(|\nabla u|^{p(x)-2}\nabla u+\mu(x)|\nabla u|^{q(x)-2}\nabla u)\cdot \nabla \varphi dx+\int_{\Omega}\alpha(x)\left(\frac{|u|^{r(x)}}{r(x)}-\gamma(x)\right)|u|^{r(x)-2}u\varphi dx\nonumber\\
&=\int_{\Omega}f\varphi dx,\,\,\ \forall \varphi\in W_0^{1,\mathcal{H}}(\Omega).
\end{align}
\end{definition}
\begin{definition}\label{Def:3.2} A function $u\in W_0^{1,\mathcal{H}}(\Omega)$ is a solution of the operator equation
\begin{equation}\label{e3.2}
\mathcal{T}(u) = f
\end{equation}
provided that for given any $f\in W_0^{1,\mathcal{H}}(\Omega)^*$
we have
\begin{align}\label{e3.3}
\langle\mathcal{T}(u),\varphi\rangle&=\int_{\Omega}(|\nabla u|^{p(x)-2}\nabla u+\mu(x)|\nabla u|^{q(x)-2}\nabla u)\cdot \nabla \varphi dx \nonumber\\
&+\int_{\Omega}\alpha(x)\left(\frac{|u|^{r(x)}}{r(x)}-\gamma(x)\right)|u|^{r(x)-2} u \varphi dx=\int_{\Omega}f\varphi dx=\langle f, \varphi\rangle,
\end{align}
for all  $\varphi \in W_0^{1,\mathcal{H}}(\Omega)$.
\end{definition}

As it is well-known from the theory of monotone operators, based on Definitions (\ref{Def:3.1}) and (\ref{Def:3.2}), demonstrating that \( u \in W_0^{1,\mathcal{H}}(\Omega) \) is a weak solution to problem (\ref{e1.1}) for every test function $\varphi \in W_0^{1,\mathcal{H}}(\Omega)$ is equivalent to solving the operator equation (\ref{e3.2}).

\begin{lemma}\label{Lem:3.1} \cite{browder1963nonlinear,minty1963monotonicity}
Let $X$ be a reflexive real Banach space. Let $A:X\rightarrow X^{\ast }$ be an (nonlinear) operator satisfying the following:

\begin{itemize}
\item[$(i)$] $A$ is coercive; that is,
$$
\lim_{\|u\|_{X}\rightarrow\infty}\frac{\langle A(u),u \rangle}{\|u\|_{X}}=+\infty.
$$
\item[$(ii)$] $A$ is hemicontinuous; that is, $A$ is directionally weakly continuous, iff the function
$$
\Upsilon(\theta)=\langle A(u+\theta w),v \rangle
$$
is continuous in $\theta$ on $[0,1]$ for every $u,w,v\in X$.
\item[$(iii)$] $A$ is monotone on the space $X$, that is, for all $u,v\in X$ we
have
\begin{equation} \label{e3.3}
\langle A\left( u\right)- A\left( v\right),u-v\rangle \geq 0.
\end{equation}
\end{itemize}
Then
\begin{equation}\label{e3.4}
A(u)=g
\end{equation}
has at least one nontrivial solution $u\in X$ for every $g\in X^{\ast }$.
\end{lemma}

The following is the first main result.

\begin{theorem}\label{Thrm:3.1} For given any $f\in W_0^{1,\mathcal{H}}(\Omega)^*$
the operator equation (\ref{e3.2}) has at least one nontrivial solution $u \in W_0^{1,\mathcal{H}}(\Omega)$ which in turn becomes a nontrivial weak solution to problem (\ref{e1.1}).
\end{theorem}

\begin{proof} First, we show that the operator $\mathcal{T}$ is coercive. We may assume that $\|u\|_{1,\mathcal{H},0}>1$.
Then using Proposition \ref{Prop:2.2a} and the necessary embeddings, it reads
\begin{align}\label{e3.5}
\langle \mathcal{T}(u),u\rangle&=\int_{\Omega}\left(|\nabla u|^{p(x)}+\mu(x)|\nabla u|^{q(x)}\right) dx +\int_{\Omega}\alpha(x)\left(\frac{|u|^{r(x)}}{r(x)}-\gamma(x)\right)|u|^{r(x)} dx\nonumber\\
&\geq \|u\|^{p^-}_{1,\mathcal{H},0}-|\gamma|_{\infty}\int_{\Omega}\alpha(x)|u|^{r(x)} dx\nonumber\\
&\geq \|u\|^{p^-}_{1,\mathcal{H},0}-c|\gamma|_{\infty}|u|_{\alpha,r(x)}\nonumber\\
&\geq \|u\|^{p^-}_{1,\mathcal{H},0}-c|\gamma|_{\infty}|u|_{1}
\end{align}
Therefore, we have $\frac{\langle \mathcal{T}(u),u\rangle }{\|u\|_{1,\mathcal{H},0}} \to +\infty$ as $\|u\|_{1,\mathcal{H},0}\to \infty$.\\
Next, we show that operator $\mathcal{T}$ is hemicontinuous. Then
\begin{align*}
&|\Upsilon(\theta_{1})-\Upsilon(\theta_{2})|
=|\langle \mathcal{T}(u+\theta_{1} w)-\mathcal{T}(u+\theta_{2} w),v \rangle|\\
& \leq \int_{\Omega}\bigg||\nabla (u+\theta_{1} w)|^{p(x)-2}\nabla (u+\theta_{1} w)-|\nabla(u+\theta_{2} w)|^{p(x)-2}\nabla (u+\theta_{2} w)\bigg||\nabla v| dx \\
&+ \int_{\Omega}\mu(x)\bigg||\nabla (u+\theta_{1} w)|^{q(x)-2}\nabla (u+\theta_{1} w)-|\nabla(u+\theta_{2} w)|^{q(x)-2}\nabla (u+\theta_{2} w) \bigg||\nabla v| dx\\
&+ \int_{\Omega}\frac{\alpha(x)}{r(x)}\bigg||u+\theta_{1} w|^{2r(x)-2}(u+\theta_{1} w)-|u+\theta_{2} w|^{2r(x)-2}(u+\theta_{2} w)  \bigg||v|dx \nonumber\\
&+\int_{\Omega}\alpha(x)\gamma(x)\bigg||u+\theta_{1} w|^{r(x)-2}(u+\theta_{1} w)-|u+\theta_{2} w|^{r(x)-2}(u+\theta_{2} w) \bigg| |v| dx.
\end{align*}
Recall the following inequality \cite{chipot2009elliptic}, for any $1<m<\infty$, there is a constant $c_m>0$ such that
\begin{equation}\label{e3.6}
(|a|^{m-2}a-|b|^{m-2}b) \leq c_m |a-b|(|a|+|b|)^{m-2},\quad \forall a,b \in \mathbb{R}^{N}.
\end{equation}
Using the H\"{o}lder inequality, Proposition \ref{Prop:2.2bb}, and the necessary embeddings we have
\begin{align*}
&|\Upsilon(\theta_{1})-\Upsilon(\theta_{2})|\\
& \leq 2^{p^+-1}|\theta_{1}-\theta_{2}| \int_{\Omega}\left(|\nabla (u+\theta_{1} w)|^{p(x)-2}+|\nabla(u+\theta_{2} w)|^{p(x)-2}\right)|\nabla w||\nabla v| dx \\
&+ 2^{q^+-1}|\theta_{1}-\theta_{2}| \int_{\Omega}\mu(x)\left(|\nabla (u+\theta_{1} w)|^{q(x)-2}+|\nabla(u+\theta_{2} w)|^{q(x)-2}\right)|\nabla w||\nabla v| dx \\
&+ 2^{2r^+-1}|\theta_{1}-\theta_{2}|\frac{|\alpha|_{\infty}}{r^-}\int_{\Omega}\left(|u+\theta_{1} w|^{2r(x)-2}+|u+\theta_{2} w|^{2r(x)-2} \right)|w||v|dx \nonumber\\
&+ 2^{r^+-1}|\theta_{1}-\theta_{2}| |\alpha|_{\infty} |\gamma|_{\infty}\int_{\Omega}\left(|u+\theta_{1} w|^{r(x)-2}+|u+\theta_{2} w|^{r(x)-2} \right)|w| |v| dx.
\end{align*}
Note that since $\theta_{1},\theta_{2} \in [0,1]$, we have $|\nabla (u+\theta_{1} w)|\leq |\nabla u|+|\nabla w|$,  $|\nabla (u+\theta_{2} w)|\leq |\nabla u|+|\nabla w|$; and $|u+\theta_{1} w|\leq |u|+| w|$, $|u+\theta_{2} w|\leq |u|+| w|$. Therefore, letting $|\nabla u|+|\nabla w|:=\xi$ and $|u|+|w|:=\eta$ in the lines above leads to
\begin{align*}
&|\Upsilon(\theta_{1})-\Upsilon(\theta_{2})|\\
& \leq 2^{p^+}|\theta_{1}-\theta_{2}| \int_{\Omega}|\xi|^{p(x)-2}|\xi||\nabla v| dx+ 2^{q^+}|\theta_{1}-\theta_{2}| \int_{\Omega}\mu(x)|\xi|^{q(x)-2}|\xi|\nabla v| dx \\
&+ 2^{2r^+}|\theta_{1}-\theta_{2}| \frac{|\alpha|_{\infty}}{r^-} \int_{\Omega}|\eta|^{2r(x)-2}|\eta||v| dx+ 2^{r^+}|\theta_{1}-\theta_{2}| |\alpha|_{\infty} |\gamma|_{\infty}\int_{\Omega}|\eta|^{r(x)-2}|\eta||v| dx \\
& \leq 2^{q^+}|\theta_{1}-\theta_{2}| \left(|\xi|_{p(x)}^{p^+-1} |\nabla v|_{p(x)}+|\mu|_{\infty}|\xi|_{q(x)}^{q^+-1} |\nabla v|_{q(x)} \right)\\
& + 2^{2r^+}|\theta_{1}-\theta_{2}| \left(\frac{|\alpha|_{\infty}}{r^-}|\eta|_{2r(x)}^{2r^+-1} |v|_{2r(x)}+|\alpha|_{\infty} |\gamma|_{\infty}|\eta|_{r(x)}^{r^+-1} |v|_{r(x)}   \right)\\
& \leq 2^{q^++2r^+}|\theta_{1}-\theta_{2}| \left(\|\xi\|_{1,\mathcal{H},0}^{p^+-1} +|\mu|_{\infty}\|\xi\|_{1,\mathcal{H},0}^{q^+-1}+\frac{|\alpha|_{\infty}}{r^-}\|\eta\|_{1,\mathcal{H},0}^{2r^+-1} +|\alpha|_{\infty} |\gamma|_{\infty}\|\eta\|_{1,\mathcal{H},0}^{r^+-1}  \right) \|v\|_{1,\mathcal{H},0}
\end{align*}
Thus, for $\theta_{1}\rightarrow \theta_{2}$ we have
$$
|\Upsilon(\theta_{1})-\Upsilon(\theta_{2})|=|\langle \mathcal{T}(u+\theta_{1} w)-\mathcal{T}(u+\theta_{2} w),v \rangle| \rightarrow 0,
$$
that is, $\mathcal{T}$ is hemicontinuous.\\
Lastly, for monotonicity of $\mathcal{T}$ if we apply the well-known inequality
\begin{equation}\label{e3.7}
\left( \left\vert x\right\vert ^{s-2}x-\left\vert y\right\vert^{s-2}y \right)\cdot\left(x-y\right) \geq 2^{-s}\left\vert x-y\right\vert^{s}, \quad \forall x,y\in \mathbb{R}^{N},\,\, s> 1,
\end{equation}
and Lemma \ref{Lem:3.3ba} together, it follows
\begin{align}\label{e3.6a}
&\langle\mathcal{T}(u)-T(v),u-v \rangle\\
&\geq \int_{\Omega}(|\nabla u|^{p(x)-2}\nabla u-|\nabla v|^{p(x)-2}\nabla v)\cdot \nabla (u-v) dx \nonumber\\
&+\int_{\Omega}\mu(x)(|\nabla u|^{q(x)-2}\nabla u-|\nabla v|^{q(x)-2}\nabla v)\cdot \nabla (u-v) dx \nonumber\\
&+\langle \mathcal{L}^{\prime}(u)-\mathcal{L}^{\prime}(v),u-v \rangle \geq 0.
\end{align}
As the conclusion of Lemma \ref{Lem:3.1}, the operator equation (\ref{e3.2}) has at least one nontrivial solution $u \in W_0^{1,\mathcal{H}}(\Omega)$ which in turn becomes a nontrivial weak solution to problem (\ref{e1.1}).
\end{proof}

\subsection{Example}
Let $\mu(x)=0$, $p(x)=r(x)=2$, $f=f(x,\psi,\nabla \psi)=-\nu \cdot \nabla \psi$, where $\nu \in \mathbb{R}^N$ is the constant velocity field describing the convection of $\psi$ (e.g., uniform, steady transport of concentration), $\alpha(x)=const>0$, and $\gamma(x)=1$. Then  problem (\ref{e1.1}) becomes
\begin{equation}\label{e4.1a}
\begin{cases}
\begin{array}{rlll}
&-\triangle \psi+\alpha\left(\frac{|\psi|^{2}}{2}-1\right)\psi=-\nu \cdot \nabla \psi \quad \text{ in }\Omega, \\
& \qquad \qquad \qquad \qquad  \quad  u=0 \qquad \qquad  \text{ on }\partial \Omega.
\end{array}
\end{cases}
\end{equation}
The model (\ref{e4.1a}) describes a Ginzburg-Landau-type equation with a convection term. When a convection term is added, the equation could model the combined effects of spatial-temporal dynamics, nonlinear diffusion, and the transport of a field driven by a flow or advection process. In this model, $\psi$ represents the macrowave function describing a superconducting state, with $|\psi|^{2}$ corresponding to the density of superconducting electrons. Originally introduced in \cite{ginzburg1950theory}, the Ginzburg-Landau equation has since played a pivotal role in advancing the understanding of macroscopic superconducting phenomena.
\begin{definition}\label{Def:4.1} A function $\psi\in W_0^{1,2}(\Omega)$ is called a weak solution to problem (\ref{e4.1a}) if
\begin{align}\label{e4.2}
&\int_{\Omega}\nabla \psi\cdot \nabla \varphi dx+\alpha\int_{\Omega}\left(\frac{|\psi|^{2}}{2}-1\right)\psi\varphi dx=\int_{\Omega}(-\nu \cdot \nabla \psi) \varphi dx,\,\,\ \forall \varphi\in W_0^{1,2}(\Omega).
\end{align}
\end{definition}

\begin{definition}\label{Def:4.2} A function $\psi\in W_0^{1,2}(\Omega)$ is a solution of the operator equation
\begin{equation}\label{e4.3}
\mathcal{T_{GL}}(\psi) = -\nu \cdot \nabla \psi
\end{equation}
if
\begin{align}\label{e4.4}
\langle\mathcal{T_{GL}}(\psi),\varphi\rangle&=\int_{\Omega}\nabla \psi\cdot \nabla \varphi dx
+\alpha\int_{\Omega}\left(\frac{|\psi|^{2}}{2}-1\right)\psi\varphi dx \nonumber\\
&=\int_{\Omega}(-\nu \cdot \nabla \psi) \varphi dx=\langle -\nu \cdot \nabla \psi , \varphi\rangle,
\end{align}
for all  $\varphi \in W_0^{1,2}(\Omega)$.
\end{definition}

\begin{theorem}\label{Thrm:4.1} The operator equation (\ref{e4.3}) has at least one nontrivial solution $u \in W_0^{1,2}(\Omega)$ which in turn becomes a nontrivial weak solution to problem (\ref{e4.1a}).
\end{theorem}

\begin{proof}
Since operator $\mathcal{T_{GL}}$ clearly satisfies the conditions $(i)-(iii)$ of Lemma \ref{Lem:3.1}, we only show that $f=-\nu \cdot \nabla \psi  \in W_0^{1,2}(\Omega)^*= W^{-1,2}(\Omega)$. \\
In doing so, we need to show that $f$ defines a bounded linear functional on $W_0^{1,2}(\Omega)$. We know, by the Riesz Representation theorem, that functional $f$ acts as
\begin{equation}\label{e4.6}
f(\varphi)=-\int_{\Omega}(\nu \cdot \nabla \psi)\varphi dx,
\end{equation}
for all test functions $\varphi \in W_0^{1,2}(\Omega)$, where $\nabla \psi \in L^{2}(\Omega)^N$ as the distributional derivative of $\psi$ since $\psi \in W_0^{1,2}(\Omega)$.\\
The linearity of $f$ follows due to the linearity of integration, and the fact that $(\nu \cdot \nabla \psi)$ is linear in $\nabla \psi$. As for the boundedness, applying the H\"{o}lder inequality and considering that the velocity field $\nu$ is constant, it reads
\begin{align}\label{e4.7}
|f(\varphi)|&=\bigg|\int_{\Omega}(\nu \cdot \nabla \psi)\varphi dx\bigg|\leq |\nu \cdot \nabla \psi|_{L^2(\Omega)}|\varphi|_{L^2(\Omega)}\leq c |\nabla \psi|_{L^2(\Omega)}|\varphi|_{L^2(\Omega)}\nonumber\\
&\leq M \|\varphi\|_{W_0^{1,2}(\Omega)}, \quad \forall \psi \in W_0^{1,2}(\Omega),
\end{align}
which implies that $f$ is bounded, where the norm of $f$ is the infimum of all such $M>0$. Thus, $f\in W^{-1,2}(\Omega)$, and therefore by Lemma \ref{Lem:3.1}, the operator equation (\ref{e4.3}) has at least one nontrivial solution $\psi \in W_0^{1,2}(\Omega)$ which in turn becomes a nontrivial weak solution to problem (\ref{e4.1a}).
\end{proof}

\section{Variational problem}

The energy functional $\mathcal{I}:W_0^{1,\mathcal{H}}(\Omega)\rightarrow \mathbb{R}$ corresponding to equation (\ref{e1.1a}) by
\begin{align}\label{e4.1}
\mathcal{I}(u)&= \int_{\Omega}\left(\frac{|\nabla u|^{p(x)}}{p(x)}+\mu(x)\frac{|\nabla u|^{q(x)}}{q(x)}\right)dx-\lambda\int_{\Omega}F(x,u)dx.
\end{align}
Let
\begin{equation}\label{e4.3aca}
\Phi(u):=\int_{\Omega}\left(\frac{|\nabla u|^{p(x)}}{p(x)}+\mu(x)\frac{|\nabla u|^{q(x)}}{q(x)}\right)dx,
\end{equation}
and also define the functional $\Psi:W_0^{1,\mathcal{H}}(\Omega) \to \mathbb{R}$ by
\begin{equation}\label{e4.3aa}
\Psi(u):=\int_{\Omega}F(x,u)dx,
\end{equation}
where $F(x,t)=\int_{0}^{t}f(x,s)ds$. Then,
\begin{equation}\label{e4.3ada}
\mathcal{I}(\cdot):=\Phi(\cdot)-\lambda\Psi(\cdot).
\end{equation}

\begin{definition}\label{Def:4.1} A function $u\in W_0^{1,\mathcal{H}}(\Omega)$ is called a weak solution to problem (\ref{e1.1a}) if
\begin{align}\label{e4.4}
&\int_{\Omega}(|\nabla u|^{p(x)-2}\nabla u+\mu(x)|\nabla u|^{q(x)-2}\nabla u)\cdot \nabla \varphi dx
=\lambda\int_{\Omega}f(x,u)\varphi dx,\,\,\ \forall \varphi\in W_0^{1,\mathcal{H}}(\Omega).
\end{align}

\end{definition}
It is well-known that the critical points of the functional $\mathcal{I}_{\lambda}$ corresponds to the weak solutions of problem (\ref{e1.1a}).

\begin{definition}\label{Def:4.2}\cite{bonanno2014existence}
Let $X$ be a real Banach space, $\Phi ,\Psi:X\rightarrow \mathbb{R}$ be two continuously G\^{a}teaux differentiable functionals and fix $r \in \mathbb{R}$. The functional $I:=\Phi-\Psi$
is said to verify the Palais–Smale condition cut off upper at $r$ (in short $(P.S.)^{[r]}$) if any sequence $(u_n)_{n \in \mathbb{N}}$ in $X$ such that
\begin{itemize}
\item [$(PS)_1^{[r]}$] $\{I(u_n)\}$ is bounded;
\item [$(PS)_2^{[r]}$] $\lim_{n \to +\infty}\|I^{\prime}(u_n)\|_{X^*}=0$;
\item [$(PS)_3^{[r]}$] $\Phi(u_n)<r$ for each $n \in \mathbb{N}$;
\end{itemize}
has a (strongly) convergent subsequence in $X$.
\end{definition}

To obtain variational solutions, we apply the following well-known critical point result given by Bonanno and Chinn\`{\i} \cite{bonanno2014existence}.

\begin{lemma}\label{Lem:4.1}\cite[Theorem 2.2]{bonanno2014existence}
Let $X$ be a real Banach space, $\Phi ,\Psi:X\rightarrow \mathbb{R}$ be two continuously G\^{a}teaux differentiable functionals such that $\inf_{x\in X}\Phi (x)=\Phi (0)=\Psi (0)=0$. Assume that there
exist $r>0$ and $\overline{x}\in X$, with $0<\Phi(\overline{x})<r$, such that:
\begin{itemize}
\item[(a1)] $\frac{1}{r}\sup_{\Phi (x)\leq r}\Psi (x)\leq \frac{\Psi (\overline{x})}{\Phi (\overline{x})}$,
\item[(a2)] \textit{for each }$\lambda \in \Lambda _{r}:=\left( \frac{\Phi (\overline{x})}{\Psi (\overline{x})},\frac{r}{\sup_{\Phi (x)\leq r}\Psi (x)}\right)$, the functional $I_{\lambda }:=\Phi -\lambda \Psi$ satisfies $(P.S.)^{[r]}$ condition.
\end{itemize}
Then, for each $\lambda \in \Lambda_{r}$, there is $x_{0,\lambda }\in \Phi ^{-1}\left( \left( 0,r\right)\right)$ such that $I_{\lambda }^{\prime }\left(x_{0,\lambda }\right) \equiv \vartheta
_{X^{\ast }}$ and $I_{\lambda }\left(x_{0,\lambda }\right) \leq I_{\lambda }\left( x\right)$ for all $x\in \Phi ^{-1}\left(\left( 0,r\right) \right) $.
\end{lemma}

\medskip
\noindent
We assume the following for the nonlinearity $f$.
\begin{itemize}
\item[$(f_{1})$] $f:\Omega \times \mathbb{R} \to \mathbb{R}$ is a Carathéodory function and there exist real parameters $\bar{c}_{1}, \bar{c}_{2}\geq0$ satisfying
\begin{equation*}
|f(x,t)| \leq  \bar{c}_{1}+ \bar{c}_{2}|t|^{s(x)-1},\\\ \forall(x,t) \in \Omega \times \mathbb{R};
\end{equation*}
\item[$(f_{2})$] $f$ is superlinear in $t$ at zero; that is,
\begin{equation*}
\limsup_{t\to 0^+} \frac{\inf_{x \in \Omega}f(x,t)t}{t^{p^-}}=+\infty;
\end{equation*}
\item[$(f_{3})$] $s \in C(\overline{\Omega})$ such that $p^+\leq s(x)<p^*(x)$.
\end{itemize}

\begin{remark}\label{Rem:4.1}
 Assumption  $(f_{2})$ implies that when $t>0$ is small enough, there exists a real parameter $\lambda_0>0$ such that
\begin{align}\label{e4.4a}
F(x,t)&=\int_{0}^{t}f(x,\xi)ds \geq \lambda_0 \int_{0}^{t} \xi^{p^--1}ds \geq \frac{\lambda_0}{p^-}t^{p^-},
\end{align}
and hence
\begin{align}\label{e4.4b}
\inf_{x \in \Omega}F(x,t)\geq \frac{\lambda_0}{p^-}t^{p^-}>0, \text{ if } t \neq 0.
\end{align}
\end{remark}

The second main result of the paper is:

\begin{theorem}\label{Thm:4.1}
Assume $(f_{1})$ and $(f_{2})$ are satisfied. Then  there exists a real parameter $\lambda_*>0$ such that for any $\lambda \in (0,\lambda_*)$ the problem (\ref{e1.1a}) admits at least one nontrivial weak solution.
\end{theorem}
\begin{proof}
Without loss of generality, we can apply Lemma \ref{Lem:4.1} for the case $r=1$.\\
We know the functional $\Phi$ is of class $C^{1}(W_0^{1,\mathcal{H}}(\Omega),\mathbb{R})$. Due to the condition $(f_1)$ and the compact embedding $W_0^{1,\mathcal{H}}(\Omega) \hookrightarrow L^{r(x)}(\Omega)$, the functional $\Psi$ is of class $C^{1}(W_0^{1,\mathcal{H}}(\Omega),\mathbb{R})$, too. Additionally, by \eqref{e4.3aca} and \eqref{e4.3aa}, it reads
\begin{equation}\label{e4.5}
\inf_{u \in W_0^{1,\mathcal{H}}}\Phi(u)=\Phi(0)=\Psi(0)=0.
\end{equation}
Note that if we define
\begin{equation}\label{e4.6}
r_{\lambda}(x)=\sup\{r_{\lambda}>0: B(x,r_{\lambda})\subseteq \Omega\}, \quad \forall x \in \Omega,
\end{equation}
and consider that $\Omega$ is open and connected in $\mathbb{R}^N$, it can easily be shown that there exists $x_0 \in \Omega$ such that $B(x_0,R)\subseteq \Omega$ with $R=\sup_{x \in \Omega}r_{\lambda}(x)$.\\
Since we need to make sure the functional $\mathcal{I}_{\lambda}$ behaves well within a restricted domain $\{u \in W_{0}^{1,\mathcal{H}}(\Omega): 0<\Phi(u)<1 \}$, we define the cut-off function  $\bar{u} \in W_{0}^{1,\mathcal{H}}(\Omega)$ by the formula
\begin{equation}\label{e4.7}
\bar{u}(x)=
\left\{
\begin{array}{ll}
0,&\text{if}\ x \in \Omega \setminus B(x_0,R),\\
r_{\lambda},&\text{if}\ x \in B(x_0,R/2),\\
\frac{2r_{\lambda}}{R}(R-|x-x_0|),&\text{if}\ x \in B(x_0,R)\setminus B(x_0,R/2):=B^*.
\end{array}
\right.
\end{equation}
For a fixed $\lambda \in (0,\lambda_*)$, from $(f_2)$ there exists
\begin{equation}\label{e4.7a}
0<r_{\lambda}<\min\left\{1,\left(\frac{p^-}{(1+|\mu|_{\infty})\omega_N\left(R^{N}-\left(\frac{R}{2}\right)^{N}\right)\left(\frac{2}{R}\right)^{\kappa_1}} \right)^{\frac{1}{p^-}}\right\}
\end{equation}
such that
\begin{align}\label{e4.7b}
\frac{p^-\inf_{x \in \Omega}F(x,r_{\lambda})}{(1+|\mu|_{\infty})(2^{N}-1)\left(\frac{2r_{\lambda}}{R}\right)^{\kappa_1}}> \frac{1}{\lambda},
\end{align}
where $\delta^{\kappa_1}=\delta^{p^-}$ if $\delta\geq1$; while $\delta^{\kappa_1}=\delta^{q^+}$ if $0<\delta< 1$, and $\delta$ is a generic placeholder for the base.
Using Remark \ref{Rem:2.1a}, we obtain
\begin{align}\label{e4.8}
\Phi(\bar{u})&\geq \frac{1}{q^+}\int_{B^*}\left(|\nabla \bar{u}|^{p(x)}+\mu(x)|\nabla \bar{u}|^{q(x)}\right) dx \geq \frac{1}{q^+}\int_{B^*}|\nabla \bar{u}|^{p(x)} dx\nonumber\\
&\geq \frac{1}{q^+}\left(\frac{2r_{\lambda}}{R}\right)^{\kappa_2}\int_{B^*}|\nabla |x-x_0||^{p(x)} dx \geq \frac{1}{q^+}\left(\frac{2r_{\lambda}}{R}\right)^{\kappa_2}\omega_N\left(R^{N}-\left(\frac{R}{2}\right)^{N}\right)\nonumber\\
& \geq \frac{1}{2^{N}q^+}\omega_N R^{N}(2^{N}-1)\left(\frac{2r_{\lambda}}{R}\right)^{\kappa_2},
\end{align}
where $\delta^{\kappa_2}=\delta^{p^-}$ if $\delta\geq1$, while $\delta^{\kappa_2}=\delta^{p^+}$ if $0<\delta< 1$; and $\omega_{N}:=\frac{\pi^{N/2}}{N/2\Gamma(N/2)}$ is the volume of the unit ball in $\mathbb{R}^N$.\\
Considering (\ref{e4.7a}), we obtain
\begin{align}\label{e4.9}
\Phi(\bar{u})&\leq \frac{1}{p^-}\omega_N\left(R^{N}-\left(\frac{R}{2}\right)^{N}\right)\left(\left(\frac{2r_{\lambda}}{R}\right)^{\kappa_3}+|\mu|_{\infty}\left(\frac{2r_{\lambda}}{R}\right)^{\kappa_4}\right)\nonumber\\
& \leq \frac{(1+|\mu|_{\infty})}{p^-}\omega_N\left(R^{N}-\left(\frac{R}{2}\right)^{N}\right)\left(\frac{2r_{\lambda}}{R}\right)^{\kappa_5}<1.
\end{align}
where $\delta^{\kappa_3}=\delta^{p^+}$, $\delta^{\kappa_4}=\delta^{q^+}$ and $\delta^{\kappa_5}=\delta^{q^+}$ if $\delta\geq1$; while $\delta^{\kappa_3}=\delta^{p^-}$, $\delta^{\kappa_4}=\delta^{q^-}$ and $\delta^{\kappa_5}=\delta^{p^-}$ if $0<\delta< 1$.\\
On the other hand, when $\bar{u}>0$ is small enough, from Remark \ref{Rem:4.1}, it reads
\begin{align}\label{e4.10}
\Psi(\bar{u})&\geq\int_{B(x_0,R/2)} F(x,\bar{u})dx\geq \inf_{x \in \Omega}F(x,r_{\lambda})\omega_N\left(\frac{R}{2}\right)^{N}.
\end{align}
Therefore, using (\ref{e4.9}) and (\ref{e4.10}) together provides
\begin{align}\label{e4.11}
\frac{\Psi(\bar{u})}{\Phi(\bar{u})}&\geq \frac{p^-\inf_{x \in \Omega}F(x,r_{\lambda})}{(1+|\mu|_{\infty})(2^{N}-1)\left(\frac{2r_{\lambda}}{R}\right)^{\kappa_1}}> \frac{1}{\lambda}.
\end{align}
For any $u \in \Phi^{-1}((-\infty, 1])$, using the embeddings $W_0^{1,\mathcal{H}}(\Omega)\hookrightarrow L^{s(x)}(\Omega)$,
$W_0^{1,\mathcal{H}}(\Omega)\hookrightarrow L^{1}(\Omega)$ and Remark \ref{Rem:2.1a}, we have
\begin{align}\label{e4.12}
\Psi(u)&\leq \int_{\Omega} \bar{c}_{1}|u|dx + \int_{\Omega} \bar{c}_{2}|u|^{s(x)}dx\leq \bar{c}_{1}c_{1}(\mathcal{H})\|u\|_{1,\mathcal{H},0} + \bar{c}_{2}c_{2}(\mathcal{H})^{s^+}\|u\|_{1,\mathcal{H},0}^{s^+}\nonumber\\
& \leq \bar{c}_{1}c_{\mathcal{H}}(q^+\Phi(u))^{\frac{1}{p^- }}+ \bar{c}_{2}c_{\mathcal{H}}^{s^+}(q^+\Phi(u))^{\frac{s^+}{p^-}}\nonumber\\
& \leq \bar{c}_{1}c_{\mathcal{H}}(q^+)^{\frac{1}{p^-}} + \bar{c}_{2}c_{\mathcal{H}}^{s^+}(q^+)^{\frac{s^+}{p^-}},
\end{align}
where $c_{\mathcal{H}}:=\max\{c_{1}(\mathcal{H}),c_{2}(\mathcal{H})\}$ and $c_{1}(\mathcal{H}),c_{2}(\mathcal{H})$ are the best embedding constants. Therefore, there exists a real parameter $\lambda_{*}>0$ such that
\begin{align}\label{e4.13}
\sup_{\Phi(u)\leq 1}\Psi(u)\leq \bar{c}_{1}c_{\mathcal{H}}(q^+)^{\frac{1}{p^-}} + \bar{c}_{2}c_{\mathcal{H}}^{s^+}(q^+)^{\frac{s^+}{p^-}}:=\frac{1}{\lambda_{*}}<\frac{1}{\lambda}.
\end{align}
In conclusion, using (\ref{e4.11}) and (\ref{e4.13}) together verifies the condition $(a1)$ of Lemma \ref{Lem:4.1}, that is
\begin{align}\label{e4.14}
\sup_{\Phi(u)\leq 1}\Psi(u)< \frac{\Psi(\bar{u})}{\Phi(\bar{u})}.
\end{align}
Next, we show that the condition $(a2)$ of Lemma \ref{Lem:4.1} is satisfied; that is, the functional $\mathcal{I}_{\lambda}$ satisfies $(P.S.)^{[1]}$ condition.\\
Let $(u_n) \in W_0^{1,\mathcal{H}}(\Omega)$ satisfy the assumptions $(PS)_1^{[1]}$-$(PS)_3^{[1]}$ but with $\|u_n\|_{1,\mathcal{H},0}\to \infty$ as $n \to \infty$. However, from Remark \ref{Rem:2.1a}, we have
\begin{align}\label{e4.15}
\|u_n\|_{1,\mathcal{H},0} \leq (q^+\Phi(u_n))^{\frac{1}{p^-}} \leq (q^+)^{\frac{1}{p^-}},
\end{align}
which is a contradiction, hence $(u_n)$ is bounded in $W_0^{1,\mathcal{H}}(\Omega)$. Thus, there is a subsequence $(u_{n})$ (not relabelled) such that
\[
\begin{split}
&u_{n}\rightharpoonup u\ (\text{weakly})  \text{ in}\ W_{0}^{1,\mathcal{H}}(\Omega),\\
&u_{n}\rightarrow u\ (\text{strongly})  \text{ in}\ L^{s(x)}(\Omega),\,\, 1\leq s(x)<p^{*}(x),\\
&u_{n}(x) \rightarrow u(x)\  \text{a.e. in}\ \Omega.
\end{split}
\]
Then by the assumptions $(PS)_1^{[1]}$ and $(PS)_2^{[1]}$, we have
\begin{align}\label{e4.16}
\langle \mathcal{I}^{\prime}_{\lambda }(u_{n}), u_{n}-u \rangle =\langle \Phi^{\prime}(u_{n}),u_{n}-u \rangle-\lambda \langle\Psi^{\prime}(u_{n}),u_{n}-u \rangle \to 0.
\end{align}
Using the H\"{o}lder inequality, the compact embeddings $W_0^{1,\mathcal{H}}(\Omega)\hookrightarrow L^{s(x)}(\Omega)$, $W_0^{1,\mathcal{H}}(\Omega)\hookrightarrow L^{1}(\Omega)$ and Proposition \ref{Prop:2.2bb}, it follows
\begin{align}\label{e4.17}
\langle\Psi^{\prime}(u_{n}), u_{n}-u \rangle & \leq \int_{\Omega}|f(x,u_{n})(u_{n}-u)|dx  \nonumber \\
& \leq \int_{\Omega}\bigg|\left(\bar{c}_{1}+ \bar{c}_{2}|u_{n}|^{s(x)-1}\right)(u_{n}-u)\bigg|dx  \nonumber \\
& \leq \bar{c}_{1} |u_{n}-u|_{1}+\bar{c}_{2}\bigg||u_{n}|^{s(x)-1}\bigg|_{\frac{s(x)}{s(x)-1}}|u_{n}-u|_{s(x)} \nonumber \\
& \leq \bar{c}_{1} |u_{n}-u|_{1}+\bar{c}_{2}|u_{n}|^{s^--1}_{s(x)}|u_{n}-u|_{s(x)}\to 0.
\end{align}
However, this implies
\begin{align}\label{e4.18}
\limsup_{n \to \infty} & \langle \Phi^{\prime}(u_{n}),u_{n}-u\rangle \nonumber \\
&=\int_{\Omega}(|\nabla u_{n}|^{p(x)-2}\nabla u_{n}+\mu(x)|\nabla u_{n}|^{q(x)-2}\nabla u_n)\cdot \nabla (u_{n}-u) dx \to 0.
\end{align}
Considering the fact that the operator $\Phi^{\prime}$ satisfies the $(S_+)$-property \cite[Theorem 3.3]{crespo2022new}, (\ref{e4.18}) implies that $u_{n} \to u\  \text{in}\ W_{0}^{1,\mathcal{H}}(\Omega)$. Therefore, condition $(a2)$ of Lemma \ref{Lem:4.1} is also verified.\\
In conclusion, by Lemma \ref{Lem:4.1}, the functional $ \mathcal{I}_{\lambda }$ has a local minimum point $\bar{u}$ in the domain $\{\bar{u} \in W_{0}^{1,\mathcal{H}}(\Omega): 0<\Phi(\bar{u})<1 \}$,
which is a nontrivial weak solution to the problem (\ref{e1.1a}).
\end{proof}
\subsection{Example}
In the following example, we shall show how the main arguments of Theorem \ref{Thm:4.1} plays out in detail.\\
Assume the following:
\begin{itemize}
\item [$\bullet$] $f(x,t)=\hat{c}_1+\hat{c}_2|t|^{s-2}t$, with $2.5\leq s< 15$ and $\max\{\hat{c}_1,\hat{c}_2\}\leq \min\{\bar{c}_{1}, \bar{c}_{2}\}$ ;
\item [$\bullet$] $N=3$;
\item [$\bullet$] $\mu(x)=1$;
\item [$\bullet$] $p(x)=2.5$;
\item [$\bullet$] $q(x)=2.8$;
\item [$\bullet$] $R=2$;
\item [$\bullet$] $r_{\lambda}=0.2$.
\end{itemize}
Then $f$ satisfies assumptions $(f_1)-(f_2)$.\\
The cut-off function  $\bar{u} \in W_{0}^{1,\mathcal{H}}(\Omega)$ becomes
\begin{equation}\label{e4.19}
\bar{u}(x)=
\left\{
\begin{array}{ll}
0,&\text{if}\ x \in \Omega \setminus B(x_0,2),\\
0.2,&\text{if}\ x \in B(x_0,1),\\
0.2(2-|x-x_0|),&\text{if}\ x \in B(x_0,2)\setminus B(x_0,1).
\end{array}
\right.
\end{equation}
The volume of the unit ball in $\mathbb{R}^3$ is $\omega_{3}=\frac{\pi^{\frac{3}{2}}}{\frac{3}{2}\Gamma(\frac{3}{2})}=\frac{4}{3}\pi$.\\
$r_{\lambda}=0.2$ satisfies the inequality (\ref{e4.7a}). Indeed,
\begin{align}\label{e4.20}
r_{\lambda}&<\min\left\{1,\left(\frac{p^-}{(1+|\mu|_{\infty})\omega_N\left(R^{N}-\left(\frac{R}{2}\right)^{N}\right)\left(\frac{2}{R}\right)^{\kappa_1}} \right)^{\frac{1}{p^-}}\right\}\nonumber\\
&=\min\left\{1,\left(\frac{2.5}{2(\frac{4}{3}\pi)(8-1)} \right)^{\frac{1}{2.5}}\right\}\approx 0.29.
\end{align}
Next, we construct the parameter $\lambda$.\\
Using Remark \ref{Rem:4.1} for $r_{\lambda}=0.2$, we obtain
\begin{align}\label{e4.21}
\inf_{x \in \Omega}F(x,0.2)\geq \frac{\lambda_0}{2.5}(0.2)^{2.5}\approx 0.007 \lambda_0.
\end{align}
Plugging the assumed  values in (\ref{e4.7b}), it reads
\begin{align}\label{e4.22}
\frac{p^-\inf_{x \in \Omega}F(x,r_{\lambda})}{(1+|\mu|_{\infty})(2^{N}-1)\left(\frac{2r_{\lambda}}{R}\right)^{\kappa_1}}\approx \frac{2.5(0.007)\lambda_0}{14 (0.2)^{2.8}}\approx 0.11\lambda_0.
\end{align}
By the Archimedean Property of $\mathbb{R}$, there exists a unique positive integer $\lambda_1$ such that $0.11\lambda_0>\frac{1}{\lambda_1}$. Hence, $0.11>\frac{1}{\lambda_0\lambda_1}:=\frac{1}{\lambda}$; or, $\lambda>\frac{1}{0.11}$.\\
Lastly, we show that there indeed exists a parameter $\lambda_{*}$ such that $\lambda \in (0,\lambda_{*})$.\\
Due to (\ref{e4.13}), it is enough to show that
\begin{align}\label{e4.23}
\hat{c}_{1}c_{\mathcal{H}}(q^+)^{\frac{1}{p^-}} + \hat{c}_{2}c_{\mathcal{H}}^{s^+}(q^+)^{\frac{s^+}{p^-}}:=\frac{1}{\lambda_*}<\frac{1}{\lambda}<0.11.
\end{align}
Using the assumed  values, we get
\begin{align}\label{e4.24}
\lambda_{*}=\frac{1}{\hat{c}_{1}c_{\mathcal{H}}(2.8)^{\frac{1}{2.5}} + \hat{c}_{2}c_{\mathcal{H}}^{s^+}(2.8)^{\frac{s^+}{2.5}}}.
\end{align}
Since the parameters $\hat{c}_{1}$ and $\hat{c}_{2}$ are arbitrary and positive, their values can be determined based on the value of the embedding constant $c_{\mathcal{H}}$ to ensure that the denominator of (\ref{e4.24}) becomes small enough such that $\lambda_{*} > \lambda>\frac{1}{0.11}$. This ensures the existence of a parameter $\lambda_{*}$ such that $\lambda \in (0,\lambda_{*})$.

\newpage
\section*{Declarations}
\section*{Data Availability}
Not applicable.
\section*{Funding}
This work was supported by Athabasca University Research Incentive Account [140111 RIA].
\section*{ORCID}
https://orcid.org/0000-0002-6001-627X

\bibliographystyle{tfnlm}
\bibliography{references}

\end{document}